\documentclass[a4paper,oneside,11pt]{article}
\usepackage[utf8]{inputenc}
\usepackage{lmodern,stackrel}
\usepackage[T1]{fontenc}
\usepackage{verbatim}
\usepackage{textcomp}
\usepackage[english]{babel}
\usepackage[pdftex]{hyperref}
\usepackage[pdftex]{color,graphicx}
\usepackage{amssymb}
\usepackage{amsmath}
\usepackage{upgreek}
\usepackage{graphicx, epstopdf}

\renewcommand{\leq}{\leqslant}
\renewcommand{\geq}{\geqslant}

\usepackage{amsmath,amsfonts,amssymb,amsthm,mathrsfs,stackrel}

\def\build#1_#2^#3{\mathrel{
\mathop{\kern 0pt#1}\limits_{#2}^{#3}}}

\theoremstyle{plain}
\newtheorem{theorem}{Theorem}

\newtheorem{corollary}{Corollary}
\newtheorem{proposition}[corollary]{Proposition}

\theoremstyle{definition}

\theoremstyle{remark}

\begin{document}
\title{Truncated long-range percolation on oriented graphs}
\author{A.\ C.\ D. van Enter\footnote{Johann Bernoulli Instituut, Rijksuniversiteit Groningen, Nijenborgh 9 9747 AG Groningen, The Netherlands}, B.\ N.\ B.\ de Lima\footnote{Departamento de Matem{\'a}tica, Universidade Federal de Minas Gerais, Av. Ant\^onio
Carlos 6627 C.P. 702 CEP 30123-970 Belo Horizonte-MG, Brazil}, D.\ Valesin$^*$}
\date{}
\maketitle

%%%-----------------------------------------------------------------------------------------------------------------------

\begin{abstract}
We consider different problems within the general theme of long-range percolation on oriented graphs. Our aim is to settle the so-called truncation question, described as follows. We are given probabilities that certain long-range oriented bonds are open; assuming that the sum of these probabilities is infinite, we ask if the probability of percolation is positive when we truncate the graph, disallowing bonds of range above a possibly large but finite threshold. We give some conditions in which the answer is affirmative. We also translate some of our results on oriented percolation to the context of a long-range contact process.
\end{abstract}
{\footnotesize Keywords: contact processes; oriented percolation; long-range percolation \\
MSC numbers:  60K35, 82B43}

\section{Introduction}

\noindent
Let $G = (\mathbb{V}(G), \mathbb{E}(G))$ be the graph with set of vertices $\mathbb{V} = \mathbb{Z}^d$ and  set of (unoriented) bonds  $\mathbb{E} = \{\langle\vec{x},\vec{x} + i\cdot \vec{e}_m\rangle: \vec{x}\in\mathbb{Z}^d,\; i \in \mathbb{Z},\; m \in \{1,\ldots, d\}\}$, where $\vec{e}_1,\ldots,\vec{e}_d$ denote the vectors in the canonical basis of $\mathbb{Z}^d$. Let $(p_i)_{i=1}^\infty$ be a sequence in the interval $[0,1]$  and consider a Bernoulli bond percolation model where each bond $e \in \mathbb{E}$ is open with probability $p_{\|e\|}$, where $\|e\|$ denotes the $l_\infty$ distance between the two endpoints of $e$.  That is, take $(\Omega, \, \mathcal{A}, \, P)$, where $\Omega = \{0,1\}^{\mathbb{E}}$, $\mathcal{A}$ is the canonical product $\sigma$-algebra, and $P = \prod_{e \in \mathbb{E}} \mu_e$, where $\mu_e({\omega}_e = 1) = p_{\|e\|} = 1- \mu_e({\omega}_e = 0)$. An element $\omega \in \Omega$ is called a percolation configuration.

As usual, the set $\{0\leftrightarrow \infty\}$ denotes the set of configurations such that the origin is connected to infinitely many vertices by paths of open bonds (bonds where $\omega_e=1$). Our principal assumption concerning the sequence $(p_i)_i$ will be
\begin{equation}\label{eq:summability}\sum_{i=1}^\infty p_i =\infty,\end{equation} so that, by Borel-Cantelli Lemma, we have $P\{0\leftrightarrow \infty\}=1$.

We now consider a truncation of the sequence $(p_i)_i$ at some finite range $k$. More precisely, for each $k> 0$ consider the truncated sequence $(p_i^k)_{i=1}^\infty$, defined by
\begin{equation}
p_i^k=\left\{
\begin{array}
[c]{l}%
p_i,\mbox{  if } i\leq k,\\
0,\ \mbox{   if } i>k.
\end{array}\right.\label{eq:truncation}
\end{equation}
and the measure $P^k = \prod_{e \in \mathbb{E}} \mu^k_e$, where
$\mu^k_e({\omega}_e = 1) = p^k_{\|e\|} = 1- \mu^k_e({\omega}_e = 0)$. Then, the \textbf{truncation question} is: does $P^k\{0\leftrightarrow \infty\}>0$ hold for $k$ large enough?

The works \cite{SSV}, \cite{Be}, \cite{FLS}, \cite{FL} and \cite{LS} (in chronological order) give an affirmative answer to the truncation question under different sets of assumptions on the dimension $d$ and  the sequence $(p_i)$. In particular, \cite{FL} gives an affirmative answer for $d\geq 3$ and no assumption on $(p_i)$ other than \eqref{eq:summability}; moreover, this work shows how the analogous question for the long-range Potts model can be studied via a long range percolation model. We would like to mention that the general truncation question for $d=2$ is still open and it is not difficult to see that for $d=1$ the answer is negative.

In the nonsummable situation, the positive answer to the truncation question (in dimensions more than 1) appears to be more robust than in the summable case. Indeed, the presence of first-order transitions in the occupation density, or in  a temperature-like parameter for summable infinite-range models, causes the truncation question to have a negative answer, as observed in \cite{FL}. Although continuity of the transition is known for Ising models, and their associated random-cluster models, in considerable generality (see for example the recent work \cite{ADCS}), this is not the case for independent percolation, where even in $d=3$ it is a famous open question in the nearest-neighbor model, while for $q$-state  Potts models first-order transitions are quite common for $q \geq 3$ (see the references \cite{BCC}, \cite{C} and \cite{GoMe}).

In  this paper, we consider the truncation question in an oriented graph. Let ${\cal G}= (\mathbb{V}({\cal G}),\mathbb{E}({\cal G}))$ be the oriented graph defined as follows. The vertex set is $\mathbb{V}({\cal G})=\mathbb{Z}^d\times\mathbb{Z}_+$, where $\mathbb{Z}_+ = \{0,1,\ldots\}$; elements of $\mathbb{V}({\cal G})$ will be denoted $(\vec{x},n)$, where $\vec{x} \in \mathbb{Z}^d$ and $n \in \mathbb{Z}_+$. The set $\mathbb{E}(\mathcal{G})$ of oriented bonds is
\begin{equation}\{\langle (\vec{x},n),(\vec{x}+i\cdot \vec{e}_m,n+1)\rangle: \vec{x} \in\mathbb{Z}^d,\;n\in\mathbb{Z}_+,\;m\in\{1,\ldots, d\},\; i\in\mathbb{Z}\}.\label{eq:bonds}\end{equation} Again we are given a sequence $(p_i)_{i=1}^\infty$ satisfying \eqref{eq:summability}
and we assume each bond  $\langle (\vec{x},n),(\vec{x}+ i\cdot \vec{e}_m,n+1)\rangle$ is open with probability $p_{i}$ independently of each other. Again denoting by $P$ the probability measure corresponding to this percolation configuration and by $\{(\vec{0}, 0)\leftrightarrow \infty\}$ the event that there exists an infinite open oriented path starting from $(\vec{0},0)$, Borel-Cantelli gives $P\{(\vec{0},0)\leftrightarrow \infty\} = 1$. For each $k> 0$, we then consider the truncated sequence given in \eqref{eq:truncation} and the corresponding measure $P^k$ and ask the truncation question, that is, whether $P^k\{(\vec{0},0)\leftrightarrow \infty\} > 0$. We prove:
\begin{theorem}\label{thm:main_cor} For any $d \geq 2$, if  the sequence $(p_i)_i$ satisfies \eqref{eq:summability}, the truncation question has an affirmative answer for the graph ${\cal G}$. Moreover, $${\displaystyle \lim_{k\to\infty}P^k\{(\vec{0},0) \leftrightarrow \infty \} = 1}.$$
\end{theorem}

This theorem is proved in the next section. In Section \ref{s:other}, we will treat a related question for the contact process and also for a different oriented graph.

\section{Proof of Proposition \ref{prop:anis}}\label{s:proof_main}
We obtain Theorem \ref{thm:main_cor} as an immediate consequence of a stronger result, which we now describe. We fix $d = 2$ and consider $\mathcal{G}$ defined as above, with vertex set $\mathbb{Z}^2\times \mathbb{Z}_+$ and set of oriented bonds given in \eqref{eq:bonds}. We take two sequences $(p_i)$, $(q_i)$ and now prescribe that bonds of the form $\langle (\vec{x}, n), (\vec{x} + i\cdot \vec{e}_1,n+1) \rangle$ are open with probability $p_i$ and bonds of the form $\langle (\vec{x}, n), (\vec{x} + i\cdot \vec{e}_2,n+1) \rangle$ are open with probability $q_i$. The truncated measure $P^k$ is obtained by truncating both sequences $(p_i)_i$ and $(q_i)_i$ at range $k$.
\begin{proposition}
\label{prop:anis}
If $(p_i)_{i=1}^\infty$ satisfies \eqref{eq:summability} and $(q_i)_{i=1}^\infty$ is not identically zero, then  ${\displaystyle \lim_{k\rightarrow\infty}P^k\{(\vec{0},0) \leftrightarrow \infty\}=1}$.
\end{proposition}

%Here we put a little bit more freedom in the model allowing anisotropy. Consider the graph ${\cal G}$ above for $d=2$ and now it is given two sequences $(p_i)_i$ and $(q_i)_i$, such that each bond $e=\langle(\vec{x},n),(\vec{y},n+1)\rangle$ is open with probability

%\begin{equation}
%\left\{
%\begin{array}
%[c]{l}%
%p_i,\mbox{  if } \vec{x}-\vec{y}=\pm i\vec{e_1},\\
%q_i,\mbox{  if } \vec{x}-\vec{y}=\pm i\vec{e_2},\\
%0,\ \mbox{   otherwise}.
%\end{array}\right.
%\end{equation}

%As before we denote by $P$ the subjacent probability measure and by $P^k$ the truncated probability measure truncated at range $k$.

%\begin{theorem}\label{thm:principal}
%If the sequence $(p_i)_i$ satisfies
%\begin{equation}\label{eq:assumption_p} \sum_{i \in \mathbb{N}} p_{i} = \infty \end{equation} and $\{i\in%\mathbb{N}:q_i>0\}\neq\emptyset$ then there exists $k$ such that $P^k\{(\vec{0},0) \leftrightarrow \infty\} > 0$. Moreover $\lim_{k\rightarrow\infty}P^k\{(\vec{0},0) \leftrightarrow \infty\}=1$.
%\end{theorem}
\begin{proof}
By assumption, we can fix $\beta\in\mathbb{N}$ such that $q_\beta > 0$.
%We define
%$$\zeta_n(\vec{x}) = I_{\{(\vec{0},0) \leftrightarrow (\vec{x},n)\}},\qquad n \in \mathbb{Z}_+,\;\vec{x} \in \mathbb{Z}^2,$$
%($I_A$ denotes the indicator function of the event $A$).
%For $n \in \mathbb{Z}_+$, define random functions $\Phi_n: \mathbb{Z} \to \mathbb{Z}$ by
%$$\Phi_n(b) = \begin{cases}\min\{a: \zeta_n(a,b) = 1\} &\text{if there exists $a$ such that $\zeta_n(a,b)=1$};\\0&\text{otherwise.} \end{cases}$$

We will define certain \textit{bifurcation events} which will imply that a point $(\vec{x},n)$ is connected to two new points $(\vec{y},n+2)$ and $(\vec{z},n+2)$. For each $(\vec{x},n) \in {\cal G}$, define the \textit{bifurcation event}
$$E_{(\vec{x},n)} = \bigcup_{a,a'\in \mathbb{Z}}\left\{\begin{array}{l}\omega_{\langle (\vec{x},n), (\vec{x} + a\vec{e}_1, n+1) \rangle} \\=\omega_{\langle (\vec{x} + a\vec{e}_1, n+1),(\vec{x}+a\vec{e}_1 + \beta \vec{e}_2, n+2) \rangle}\\=\omega_{\langle (\vec{x} + a\vec{e}_1, n+1),(\vec{x}+a\vec{e}_1 + a'\vec{e}_1, n+2) \rangle}=1 \end{array} \right\}.$$
We have
\begin{equation*}
\label{eq:probE}
P^k(E_{(\vec{x},n)}) = 1 - \prod_{a: |a|\leq k}\left( 1 - p_{|a|} \cdot q_{\beta} \cdot \left( 1- \prod_{a': |a'|\leq k} (1-p_{|a'|})\right)\right)=\gamma_k,
\end{equation*}
which can be made arbitrarily close to 1 by increasing $k$, by \eqref{eq:summability}.
Also note that
\begin{equation}
\label{eq:inclusion_E}
\{(\vec{0},0) \leftrightarrow (\vec{x},n)\} \cap E_{(\vec{x},n)} \subseteq \bigcup_{a, a'\in \mathbb{Z}_+} \left\{\begin{array}{l}(\vec{0},0) \leftrightarrow (\vec{x} + a \vec{e}_1 + a'\vec{e}_1, n+2),\\ (\vec{0},0) \leftrightarrow (\vec{x} + a \vec{e}_1 + \beta\vec{e}_2, n+2) \end{array}\right\}.
\end{equation}
Finally, under $P$ and $P^k$, $E_{((a,b),m)}$ and $E_{((a',b'),n)}$ are independent and identically distributed as soon as either $b\neq b'$ or $|m-n| \geq 2$.

The next step is to prove that, if $k$ is large enough, a certain projection of the $k$-truncated process dominates an oriented supercritical Bernoulli percolation on $\mathbb{Z}^2_+$. Define the following order in $\mathbb{Z}^2_+$: given $(m_1,n_1),(m_2,n_2)\in\mathbb{Z}^2_+$ we say that $(m_1,n_1)\prec(m_2,n_2)$ if and only if $n_1<n_2$ or $(n_1=n_2\mbox{ and }m_1<m_2)$. Given $X\subset\mathbb{Z}^2_+$, we define the exterior boundary of $X$ as the set $$\partial_e X=\{(m,n)\in \mathbb{Z}^2_+\backslash X: (m,n-1)\in X\mbox{ or }(m-1,n-1)\in X\}.$$
We define the vertex $(m,n)\in\mathbb{Z}^2_+$ as {\em red} if and only if the following event occurs: $\bigcup_{a\in\mathbb{Z}}\left(\{(\vec{0},0) \leftrightarrow((a,m\beta),2n) \}\cap E_{((a,m\beta),2n)}\right)$.

\begin{figure}[htb]
\begin{center}
\setlength\fboxsep{0pt}
\setlength\fboxrule{0pt}
\fbox{\includegraphics[width = 0.9\textwidth]{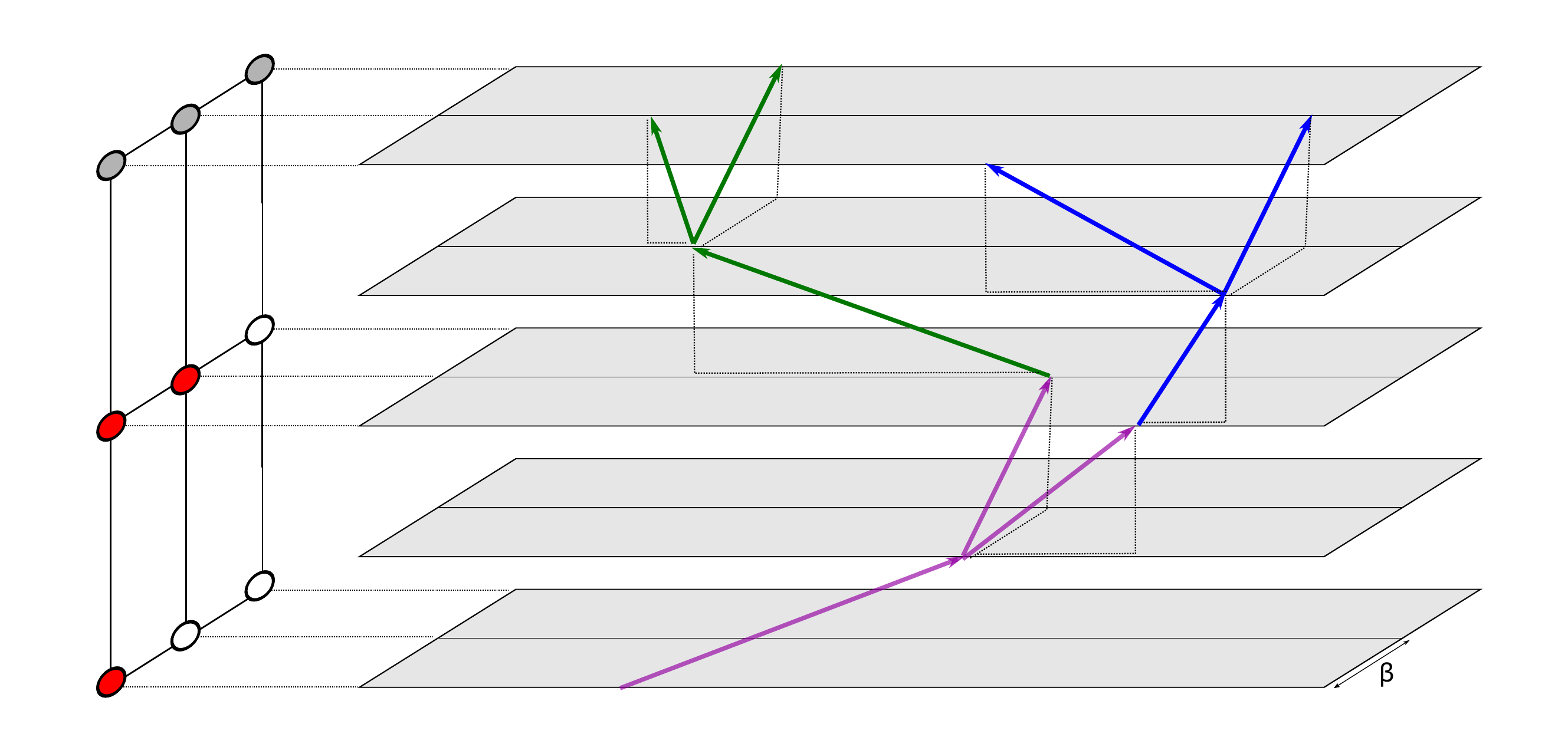}}
\caption{The occurrence of each bifurcation event is represented by a triple of arrows with the same color. On the left side of the picture, we represent a certain projection which will be defined from these events: red vertices will appear at the (projected) starting points of bifurcations. With the information available in the picture, it is impossible to tell whether or not the three vertices on top are red.}
\end{center}
\end{figure}

To avoid confusion, let us emphasize that, if a vertex in $\mathbb{Z}^2_+$ has coordinates $(m,n)$, then this vertex is defined as red through an event in the original lattice $\mathbb{Z}^2 \times \mathbb{Z}_+$; this event involves a bifurcation with some starting point in the line $\{((a, m\beta), 2n): a \in \mathbb{Z} \}$.  In particular, in Figure 1, one horizontal unit and one vertical unit in the lattice depicted on the left correspond respectively to $\beta$ units and $2$ units in the lattice on the right.

We will construct a red cluster dynamically, defining inductively two sequences $(A_i)_i$ and $(B_i)_i$ of subsets of $\mathbb{Z}^2_+$. Set $A_0=B_0=\emptyset$ and $x_0=(0,0)$. Assuming $A_j,B_j$ and $x_j$ have been defined for $j=0,\dots,i$, we let
\begin{align*}
&A_{i+1} =
\begin{cases}
A_i\cup\{x_i\},&\mbox{  if } x_i\mbox{ is red},\\
A_i, &\mbox{   otherwise},
\end{cases}\qquad B_{i+1}=\begin{cases}
B_i,&\mbox{  if } x_i\mbox{ is red},\\
B_i\cup\{x_i\},& \mbox{  otherwise.}
\end{cases}
\end{align*}
Now, if $(\partial_e A_{i+1})\backslash B_{i+1} = \emptyset,$ we stop our recursive definition. Otherwise we let $x_{i+1}$ be the minimal point of $(\partial_e A_{i+1})\backslash B_{i+1}$ with respect to the order $\prec$ defined above, and continue the recursion. Regardless of whether or not the recursion ever ends, we let    ${\cal C}$ be the union of all sets $A_i$ that have been defined. It follows from \eqref{eq:inclusion_E} that $\{|{\cal C}| =\infty\} \subseteq \{(\vec{0},0) \leftrightarrow \infty\}$.

Now, observe that
$$P^k(x_{i}\mbox{ is red} \mid (A_j,B_j): 0 \leq j \leq i)\geq\gamma_k.$$
This implies that ${\cal C}$ stochastically dominates the cluster of the origin in Bernoulli oriented site percolation on $\mathbb{Z}^2_+$ with parameter $\gamma_k$ (see Lemma 1 of \cite{GM}). As noted earlier, $\gamma_k$ can be made arbitrarily close to 1; this proves that ${\displaystyle \lim_{k\to\infty} P^k(|{\cal C}|=\infty) = 1}$.
\end{proof}

\section{Contact process and oriented percolation on other graphs}
\label{s:other}
\subsection{The Contact Process}\label{contato}
Here we will give a counterpart of Theorem \ref{thm:main_cor} for the contact process obtained from truncating an infinite set of rates. Let us define precisely the model that we have in mind. We are given a sequence of non-negative real numbers, $(\lambda_i)_{i=1}^\infty$. We take a family of independent Poisson point processes on $[0,\infty)$:
\begin{itemize}
\item a process $D^{\vec{x}}$ of rate 1 for each $\vec{x} \in \mathbb{Z}^d$;
\item a process $B^{(\vec{x},\vec{y})}$ of rate $\lambda_{|i|}$ for each ordered pair $(\vec{x}, \vec{y})$ with $\vec{x} \in \mathbb{Z}^d$ and $\vec{y} = \vec{x} + i\cdot \vec{e}_m$ with $i \in \mathbb{Z}$ and $m \in \{1,\ldots, d\}$.
\end{itemize}
We view each of these processes as a random discrete subset of $[0,\infty)$ and write, for $0\leq a < b$, $D^{\vec{x}}_{[a,b]} = D^{\vec{x}} \cap [a, b]$ and $B^{(\vec{x},\vec{y})}_{[a,b]} = B^{(\vec{x},\vec{y})} \cap [a,b]$.

Fix $k \in \mathbb{N}$. Given $\vec{x}, \vec{y} \in \mathbb{Z}^d$ and $0 \leq s \leq t$, we say $(\vec{x},s)$ and $(\vec{y},t)$ are $k$-connected, and write $(\vec{x},s) \stackrel{k}{\leftrightarrow} (\vec{y},t)$, if there exists a function $\gamma:[s,t] \to \mathbb{Z}^d$ that is right-continuous, constant between jumps and satisfies:
\begin{align*}\begin{array}{ll}\gamma(s) = \vec{x},\;\gamma(t) = \vec{y} \;\text{ and, for all } r \in [s,t], &\gamma(r) \notin D^{\gamma(r)},\\ & r \in B^{(\gamma(r-),\gamma(r))} \text{ if } \gamma(r) \neq \gamma(r-),\\
&|\gamma(r) - \gamma(r-)| \leq k .\end{array} \end{align*}
We then define
$$\xi_{t,k}(\vec{x}) = I\{(\vec{0},0) \stackrel{k}{\leftrightarrow} (\vec{x},t)\},\qquad \vec{x} \in \mathbb{Z}^d,\; t \geq 0.$$ $(\xi_{t,k})_{t \geq 0}$ is then a Markov process on the space $\{0,1\}^{\mathbb{Z}^d}$ for which the configuration that is identically equal to 0 (denoted here by $\underline{0}$) is absorbing. In case $\lambda_i > 0$ only for $i =1$, $(\xi_{t,1})$ is the contact process of Harris (\cite{harris}).

\begin{theorem}
For all $d\geq 2$, if $\sum_{i=1}^\infty \lambda_i = \infty$, then $$ \lim_{k \to \infty} P\left(\xi_{t,k} \neq \underline{0} \text{ for all } t \right) = 1.$$
\end{theorem}
\begin{proof}
It is enough to prove the case $d=2$. Fix $\delta > 0$ and $k \in \mathbb{Z}_+$. Let $t_n = n\delta$, for $n \in \{0,1,\ldots\}$. Fix $b$ such that $\lambda_{b} > 0$.

For $\vec{x}\in \mathbb{Z}^d$ and $n \in \mathbb{Z}_+$, let $F_{(\vec{x}, n)}$ be the event
$$\{D^{\vec{x}}_{[t_n,t_{n+1}]} = \varnothing\}\cap \bigcup_{a \in \mathbb{Z}} \left\{\begin{array}{l} D^{\vec{x} + a \vec{e}_1}_{[t_n,t_{n+1}]} = D^{\vec{x} + a \vec{e}_1 + b \vec{e}_2}_{[t_n,t_{n+1}]} = \varnothing,\\[.4cm] B^{(\vec{x}, \vec{x} + a \vec{e}_1)}_{[t_n, t_n + \delta/2]}\neq \varnothing,\;  B^{(\vec{x}+a\vec{e}_1, \vec{x} + a \vec{e}_1 + b\vec{e}_2)}_{[t_n+\delta/2, t_{n+1}]} \neq \varnothing \end{array} \right\}.$$
Then,
$$P^k(F_{(\vec{x},n)}) = e^{-\delta} \left(1- \prod_{a=-k}^k \left(1- e^{-2\delta}\cdot (1-e^{-\frac{\lambda_{|a|} \delta}{2}}) \cdot (1-e^{-\frac{\lambda_{|b|}\delta}{2}})\right)\right).$$
By first taking $\delta$ small and then taking $k$ large, the probability of these events can be made arbitrarily close to 1. Moreover,
$$\{\xi_{t_n,k}(\vec{x}) = 1\} \cap F_{(\vec{x},n)} \subseteq \bigcup_{a \in \mathbb{Z}}\left\{\begin{array}{l}\xi_{t_{n+1},k}(\vec{x}+a\vec{e}_1) \\=\xi_{t_{n+1},k}(\vec{x} + a \vec{e}_1 + b \vec{e}_2) =1\end{array} \right\}.$$

The proof is then completed with a comparison with oriented percolation almost identical to the one that established Proposition \ref{prop:anis}.
\end{proof}

\subsection{Other Oriented Graphs}
In this section we consider a graph ${\cal G^*}=(\mathbb{V}({\cal G^*}),\mathbb{E}({\cal G^*}))$. Once more, the vertex set is $\mathbb{V}({\cal G^*})=\mathbb{Z}^d\times\mathbb{Z}_+,\ d\geq 1$. The set of bonds  $\mathbb{E}({\cal G^*})$ consists of two disjoint subsets; one of them, denoted $\mathbb{E}_v$, only contains oriented bonds, and the other, $\mathbb{E}_h$, only unoriented bonds. These subsets are given by
\begin{align*}&\mathbb{E}_v=\{\langle (\vec{x},n),(\vec{x},n+1)\rangle: \vec{x}\in\mathbb{Z}^d,n\in\mathbb{Z}_+\},\\&\mathbb{E}_h=\{\langle\vec{x},n),(\vec{x}+i\cdot \vec{e}_m,n)\rangle: \vec{x}\in\mathbb{Z}^d,\;n\in\mathbb{Z}_+,\;i\in\mathbb{Z},\;m\in\{1,\dots,d\}\}.\end{align*}
That is, we are considering  the hypercubic lattice where there are nearest neighbour, oriented bonds along the vertical direction and long range, unoriented bonds parallel to all other coordinate axes.

We consider an anisotropic oriented Bernoulli percolation on this graph. Given $\epsilon\in(0,1)$ and a sequence $(p_i)_{i=1}^\infty$ in the interval $[0,1]$, each bond $e\in\mathbb{E}$ is open with probability $\epsilon$ or $p_{\|e\|}$, if $e\in\mathbb{E}_v$ or $e\in\mathbb{E}_h$, respectively.

Given two vertices $(\vec{x},m)$ and $(\vec{y},n)$ with $m < n$, we say that $(\vec{x},n)$ and $(\vec{y},m)$ are connected if there exists a path $$\langle (\vec{x},n)=(\vec{x}_0,n_0),(\vec{x}_1,n_1),\dots,(\vec{x}_s,n_s)=(\vec{y},m)\rangle$$ such that $\langle (\vec{x}_i,n_i),(\vec{x}_{i+1},n_{i+1})\rangle\in\mathbb{E}_h$ or ($\vec{x}_i=\vec{x}_{i+1}$ and $n_{i+1}=n_i+1$) for all $i=0,\dots,s-1$, and the bonds $\langle (\vec{x}_i,n_i),(\vec{x}_{i+1},n_{i+1})\rangle$ are open for all $i=0,\dots,s-1$. That is, all allowed paths use vertical bonds only in the upward direction. We use the notation $\{(\vec{0},0)\stackrel{*}{\leftrightarrow}\infty\}$ to denote the set of configurations in which there is an infinite open path starting at $(\vec{0},0)$. We use also the notations $P$ and $P^k$ to denote the non-truncated and the truncated (in the range $k$) probability measures, respectively.

\begin{theorem}\label{hex} For any $d\geq 2$, any $\epsilon >0$ and any sequence $(p_i)_{i=1}^\infty$ such that $\sum_{i\in\mathbb{N}} p_i =\infty$, we have ${\displaystyle \lim_{k\rightarrow\infty}P^k\{(\vec{0},0)
\stackrel{*}{\leftrightarrow} \infty\}=1}$.
\end{theorem}
A weaker result was proven in \cite{FL} (see Theorem 6 therein) in the context of non-oriented and isotropic percolation. The proof of Theorem \ref{hex} is inspirated by the proof thereof (\cite{FL}).

\begin{proof}
It is sufficient to prove the theorem for $d = 2$.

Let $\upgamma: \mathbb{Z} \to \mathbb{Z}^2$ be the function satisfying
$$\upgamma(0) = \vec{0}, \qquad \upgamma(m+1)- \upgamma(m) = \begin{cases} \vec{e}_1&\text{if $m$ is even,}\\-\vec{e}_2&\text{if $m$ is odd.}\end{cases}$$

Define the events
$$H_{m,n} = \left\{\begin{array}{l}\text{$(\upgamma(m),n)$ and $(\upgamma(m+1),n)$ are connected }\\
\text{by a path of open bonds of $\mathbb{E}_h$ that is }\\\text{entirely contained in the line that contains}\\\text{$(\upgamma(m),n)$ and $(\upgamma(m+1),n)$}\end{array}\right\},\; m \in \mathbb{Z},\;n\in \mathbb{Z}_+.$$
Clearly, $P^k(H_{m,n}) = P^k(H_{0,0})$ for all $m,n$. Also note that, if $(m_1, n_1) \neq (m_2,n_2)$, then the line that contains $(\upgamma(m_1),n_1)$ and $(\upgamma(m_1 + 1), n_1)$ does not share any bonds of $\mathbb{E}_h$ with the line that contains $(\upgamma(m_2),n_2)$ and $(\upgamma(m_2+1), n_2)$. Hence, the events $H_{m,n}$ defined above are independent. Moreover, we have
\begin{equation}\label{eq:one_dimension}
\lim_{k \to \infty} P^k(H_{m,n}) = 1
\end{equation}
(a proof of this can be found in the first few lines of the proof of Theorem 6 in \cite{FL}).

Now, fix $\epsilon > 0$ and $\delta > 0$. Let $N$ be an integer satisfying $(1-(1-\epsilon)^{N})^2 > 1-\delta/2$. Then, using \eqref{eq:one_dimension}, choose $k > 0$ such that $(P^k(H_{0,0}))^{2N} > 1 - \delta/2$. Then let
$$\Lambda_0 = \left\{(a,n)\in \mathbb{Z} \times \mathbb{Z}_+: a + n \text{ is even} \right\}.$$
For each $(a,n) \in \Lambda_0$, let $\zeta(a,n)$ be the indicator function of the event
\[\left(\bigcap_{m=aN}^{aN + 2N -1}H_{m,n} \right) \cap \left( \bigcup_{m=aN}^{aN + N -1} \left\{\langle (\Gamma(m),n),(\Gamma(m),n+1)\rangle \text{ is open}\right\} \right)\cap\]
\[\left( \bigcup_{m=aN+N}^{aN + 2N -1} \left\{\langle (\Gamma(m),n),(\Gamma(m),n+1)\rangle \text{ is open}\right\} \right).\]

Then, the elements of the sequence of random variables $(\zeta(a,n))_{(a,n) \in \Lambda_0}$ are independent and, by the choice of $N$, each of them is equal to 1 with probability $1-\delta$. Now note that an infinite sequence $(a_i)_{i =0}^\infty$ such that $a_0=0$, $|a_{i+1} - a_{i}| = 1$ and $\zeta(a_i,i) = 1$ for each $i$ necessarily corresponds to an infinite open path in $G$. Moreover, the probability of existence of such a sequence can be taken arbitrarily close to 1 since $\delta$ is arbitrary.
\end{proof}

\section*{Acknowledgements}

This work was done during B.N.B.L.'s sabbatical stay at IMPA; he would like to thank Rijksuniversiteit Groningen and IMPA for their hospitality. The research of B.N.B.L. was supported in part by CNPq and FAPEMIG (Programa Pesquisador Mineiro).

\end{document}